\documentclass{article}

\usepackage{amssymb}
\usepackage{amsfonts}
\usepackage{amsmath}
\usepackage{amsthm}
\usepackage{ascmac}
\usepackage[all]{xy}
\usepackage[colorlinks=true, urlcolor=rltblue, citecolor=drkgreen, linkcolor=drkred] {hyperref}
\usepackage {hyperref}
\usepackage{color}
\definecolor{rltblue}{rgb}{0,0,0.4}
\definecolor{drkgreen}{rgb}{0,0.4,0}
\definecolor{drkred}{rgb}{0.5,0,0}
\usepackage{pdfsync}
\usepackage[affil-it]{authblk}

\newtheorem{Theorem}{Theorem}[section]
\newtheorem{Lemma}[Theorem]{Lemma}

\newtheorem{Observation}{Observation}
\newtheorem{Fact}{Fact}
\newtheorem{Question}{Question}
\numberwithin{equation}{section}
\newtheorem{Claim}{Claim}

\theoremstyle{definition}

\newtheorem*{Acknowledgement}{Acknowledgement}

\theoremstyle{plain}

\newcounter{contenumi}

\def\upto{\mathop{\upharpoonright}}

\def\and{\mathrel{\&}}

\def\om{\omega}

\def\om{\omega}

\newcommand{\interval}{\mathbb{I}}

\newcommand{\pcolon}{\colon\!\!\!\subseteq}

\begin{document}

\title{The Brouwer invariance theorems in reverse mathematics}
\author{Takayuki Kihara\thanks{email address: \texttt{kihara@i.nagoya-u.ac.jp}
}}
\affil{\small Graduate School of Informatics, Nagoya University, Japan}
\date{}

\maketitle

\begin{abstract}
In his book \cite{Stillwell}, John Stillwell wrote ``finding the exact strength of the Brouwer invariance theorems seems to me one of the most interesting open problems in reverse mathematics.''
In this article, we solve Stillwell's problem by showing that (some forms of) the Brouwer invariance theorems are equivalent to weak K\"onig's lemma over the base system ${\sf RCA}_0$.
In particular, there exists an explicit algorithm which, whenever weak K\"onig's lemma is false, constructs a topological embedding of $\mathbb{R}^4$ into $\mathbb{R}^3$.
\end{abstract}

\section{Introduction}

How different are $\mathbb{R}^m$ and $\mathbb{R}^n$?
It is intuitively obvious that $\mathbb{R}^m$ and $\mathbb{R}^n$ are not homeomorphic whenever $m\not=n$.
However, it is not as easy as it appears.
Quite a few prominent mathematicians tried to solve this {\em invariance of dimension} problem, and nobody before Brouwer could succeed to provide a correct rigorous proof (see \cite[Section 5.1]{vanDalen} for the history of the invariance of dimension problem).


In the early days of topology, Brouwer proved three important theorems: {\em The Brouwer fixed point theorem}, {\em the invariance of dimension theorem}, and {\em the invariance of domain theorem}.
Modern proofs of these theorems make use of singular homology theory \cite{Hatcher} or its relative of the same nature, 
but even today, no direct proof (only using elementary topology) has been found.

Brouwer's intuitionistic standpoint eventually led him to refuse his theorems, and even propose a ``counterexample'' to his fixed point theorem.
As an alternative, Brouwer introduced an approximate version of the fixed point theorem (which follows from Sperner's lemma); however it does not provide us an approximation of an actual fixed point as already pointed out by Brouwer himself, cf.\ \cite[p.\ 503]{vanDalen}.
(Indeed, there is {\bf no} computable algorithm which, given a sequence $(x_n)_{n\in\mathbb{N}}$ of points such that $x_n$ looks like a fixed point within precision $2^{-n}$, produces an approximation of an actual fixed point.)
Then, how non-constructive are Brouwer's original theorems?

We examine this problem from the perspective of reverse mathematics.
Reverse mathematics is a program to determine the exact (set-existence) axioms which are needed to prove theorems of ordinary mathematics.
We employ a subsystem ${\sf RCA}_0$ of second order arithmetic as our base system, which consists of Robinson arithmetic (or the theory of the non-negative parts of discretely ordered rings), $\Sigma^0_1$-induction schema, and $\Delta^0_1$-comprehension schema, cf.\ \cite{Simpson,Stillwell}.

Roughly speaking, the system ${\sf RCA}_0$ corresponds to computable mathematics, which has enough power to show the approximate fixed point theorem (cf.\ \cite[Section IV.7]{Simpson}).
On the other hand, Orevkov \cite{Orevkov} showed that the Brouwer fixed point theorem is invalid in computable mathematics in a rigorous sense; hence ${\sf RCA}_0$ is not enough for proving the actual fixed point theorem.
%
%
%
%

In the Bishop-style constructive mathematics, it is claimed that a uniform continuous version of the invariance of dimension theorem has a constructive proof (cf.\ Beeson \cite[Section I.19]{Bee85}).
Similarly, in the same constructive setting, Julian-Mines-Richman \cite{JMR83} studied the Alexander duality theorem and the Jordan-Brouwer separation theorem (which are basic tools to show the invariance of domain theorem in modern algebraic topology, cf.\ \cite{Hatcher}).
However, these constructive versions are significantly different from original ones (from constructive and computable viewpoints).

Concerning the original theorems, Shioji-Tanaka \cite{ShTa90} (see also \cite[Section IV.7]{Simpson}) utilized Orevkov's idea to show that, over ${\sf RCA}_0$, the Brouwer fixed point theorem is equivalent to {\em weak K\"onig's lemma} (${\sf WKL}$): Every infinite binary tree has an infinite path.
%
%
%
Other examples equivalent to ${\sf WKL}$ include the Jordan curve theorem and the Sch\"onflies theorem  \cite{SaYo07}.

In his book \cite{Stillwell}, John Stillwell wrote ``finding the exact strength of the Brouwer invariance theorems seems to me one of the most interesting open problems in reverse mathematics.''
%
%
In this article, we solve this problem by showing that some forms of the Brouwer invariance theorems are equivalent to weak K\"onig's lemma over the base system ${\sf RCA}_0$.

\begin{Theorem}\label{thm:main-theorem}
The following are equivalent over ${\sf RCA}_0$:
\begin{enumerate}
\item Weak K\"onig's lemma.
\item (Invariance of Domain) Let $U\subseteq\mathbb{R}^m$ be an open set, and $f\colon U\to\mathbb{R}^m$ be a continuous injection.
Then, the image $f[U]$ is also open.
\item (Invariance of Dimension I) If $m>n$ then there is no continuous injection from $\mathbb{R}^m$ into $\mathbb{R}^n$.
\item (Invariance of Dimension II) If $m>n$ then there is no topological embedding of $\mathbb{R}^m$ into $\mathbb{R}^n$.
\end{enumerate}
\end{Theorem}

\begin{proof}
For (1)$\Rightarrow$(2), as mentioned in Stillwell \cite{Stillwell}, the usual algebraic topology machineries (cf.\ \cite{Hatcher}) are available in ${\sf WKL}_0$.
A simpler proof of the invariance of domain theorem is presented in Tao \cite[Section 6.2]{Tao}, which can also be carried out in ${\sf WKL}_0$.

For (2)$\Rightarrow$(3), suppose $m>n$ and that there is a continuous injection $f$ from $\mathbb{R}^m$ into $\mathbb{R}^n$.
Define $g\colon\mathbb{R}^m\to\mathbb{R}^m$ by $g(x)=(f(x),0,0,\dots,0)$.
Then, $g$ is also a continuous injection.
Hence, by invariance of domain, the image of $g$ is open.
However, if $m>n$, then $\{(z,0,0,\dots,0)\in\mathbb{R}^m:z\in\mathbb{R}^n\}$ does not contain a nonempty open set.
Thus, we get $m\leq n$.

The implication (3)$\Rightarrow$(4) is obvious.
We devote the rest of the paper to proving the implication (4)$\Rightarrow$(1).
\end{proof}

We first describe the outline of our strategy for (the contrapositive of) (4)$\Rightarrow$(1):

First, we will show that several basic results in topological dimension theory are provable in ${\sf RCA}_0$.
More explicitly, ${\sf RCA}_0$ proves that, whenever the $n$-sphere $\mathbb{S}^n$ is an absolute extensor for $X$, the covering dimension of $X$ is at most $n$.
We also show that the N\"obeling imbedding theorem (stating that every $n$-dimensional Polish space is topologically embedded into a ``universal'' $n$-dimensional subspace of $\mathbb{R}^{2n+1}$) is provable in ${\sf RCA}_0$.

Then, under ${\sf RCA}_0+\neg{\sf WKL}$, we will show that the $1$-sphere $\mathbb{S}^1$ is an absolute extensor (for all Polish spaces).
This means that, under $\neg{\sf WKL}$, {\em every Polish space is at most one-dimensional}, and therefore, by the N\"obeling imbedding theorem, every Polish space is topologically embedded into $\mathbb{R}^3$.
In particular, we will see that, assuming $\neg{\sf WKL}$, a topological embedding of $\mathbb{R}^4$ into $\mathbb{R}^3$ {\bf does} exist.
%
%
%
However, the following two questions remain open.
 
\begin{Question}
Does ${\sf RCA}_0$ prove that there is no topological embedding of $\mathbb{R}^3$ into $\mathbb{R}^2$?
\end{Question}

\begin{Question}
Does ${\sf RCA}_0$ prove that $\mathbb{R}^m$ is not homeomorphic to $\mathbb{R}^n$ whenever $m\not=n$?
\end{Question}

\subsection{Preliminaries}

We assume that the reader is familiar with reverse mathematics (cf.~Stillwell \cite{Stillwell} and Simpson \cite{Simpson}).
In particular, we use standard formulations of mathematical concepts in second order arithmetic:
A real number is coded as a Cauchy sequence of rational numbers with modulus of convergence (\cite[Definition II.4.4]{Simpson}).
A Polish space $X$ is coded as a pair of a countable set $A\subseteq\mathbb{N}$ (which represents a countable dense subset of a space $X$) and a function $d\colon A^2\to\mathbb{R}$ (\cite[Definition II.5.2]{Simpson}).
A code of an open set $U\subseteq X$ is any sequence of rational open balls $B_n$ whose union is $U$ (\cite[Definition II.5.6]{Simpson}).
A code of a partial continuous function $f\pcolon X\to Y$ is any data $\Phi$ specifying a modulus of pointwise continuity for $f$; that is, if $(a,r,b,s)$ is enumerated into $\Phi$ at some round, then $x\in{\rm dom}(f)$ and $d_X(x,a)<r$ implies $d_Y(f(x),b)\leq s$ (\cite[Definition II.6.1]{Simpson}).
A topological embedding $f$ of $X$ into $Y$ is coded as a pair of (codes of) continuous functions $(f,g)$ such that $g\circ f(x)=x$ for any $x\in X$.

In particular, we note that a ``code'' of some mathematical object can always be considered as an element of $\mathbb{N}^\mathbb{N}$.
In reverse mathematics, we often use sentences like ``for a given $x$ one can {\em effectively} find a $y$ such that $\dots$'' when there is a partial continuous function $f\pcolon\mathbb{N}^\mathbb{N}\to\mathbb{N}^\mathbb{N}$ such that if $\dot{x}$ is a code of $x$ then $f(\dot{x})$ is defined and returns a code of such a $y$.

\section{Proof of (4)$\Rightarrow$(1)}

\subsection{Coincidence of dimension}

In this section, we discuss a few basic results on topological dimension theory within ${\sf RCA}_0$.
For basics on classical topological dimension theory, see Engelking \cite{Engelking} and Nagata \cite{Nagata}.

It is not hard to see that the results we will discuss in this section are provable within ${\sf RCA}$ (i.e., ${\sf RCA}_0$ plus full induction); however, most basic results in topological dimension theory involve induction argument (see Lemma \ref{lem:coceshrinking} and Lemma \ref{lem:lemma2}), so we will need a few tricks to make the proofs work only with $\Sigma^0_1$-induction.

\subsubsection{Normality}

A space $X$ is {\em normal} if for any (negative codes of) disjoint closed sets $P_0,P_1\subseteq X$, one can find (positive codes of) disjoint open sets $S_0,S_1\subseteq X$ such that $P_0\subseteq S_0$ and $P_1\subseteq S_1$.
A space $X$ is {\em perfectly normal} if for any disjoint closed sets $P_0,P_1\subseteq X$, one can effectively find a (code of) continuous function $g\colon X\to[0,1]$ such that for all $x\in X$ and $i<2$, $x\in C_i$ if and only if $g(x)=i$.
Note that we require effectivity for all notions to reduce the complexity of induction involved in our proofs.
It is known that the effective version of Urysohn's lemma is provable within ${\sf RCA}_0$ as follows:

\begin{Fact}[cf.\ Simpson {\cite[Lemma II.7.3]{Simpson}}]\label{fact:perfectly-normal}
Over ${\sf RCA}_0$, every Polish space is perfectly normal.
\qed
\end{Fact}

Let $\mathcal{U}$ be a cover of a space $X$.
A cover $\mathcal{V}$ of $X$ is a {\em refinement of} $\mathcal{U}$ if for any $B\in\mathcal{V}$ there is $A\in\mathcal{U}$ such that $B\subseteq A$.
A {\em shrinking of a cover $\mathcal{U}=(U_i)_{i<s}$} of $X$ is a cover $\mathcal{V}=(V_i)_{i<s}$ of $X$ such that $V_i\subseteq U_i$ for any $i<s$.

\begin{Lemma}[${\sf RCA}_0$]\label{lem:coceshrinking}
Let $X$ be a perfectly normal space.
Then, for every finite open cover $\mathcal{U}$ of $X$, one can effectively find a closed shrinking of $\mathcal{U}$. 
\end{Lemma}

\begin{proof}
Let $\mathcal{U}=\{U_i\}_{i<k}$ be a finite open cover.
By perfect normality of $X$, for each $i<k$ one can effectively find a continuous function $g_i\colon X\to[0,1]$ such that $g_i^{-1}(x)>0$ iff $x\in U_i$ for any $x\in X$.
One can effectively construct (a code of) the following sequence $\langle g'_i,\tilde{g}_i\rangle_{i<k}$ of (possibly partial) continuous functions:
\begin{align*}
\tilde{g}_i(x)&=\frac{g_i(x)}{g_i(x)+\max\{g_s'(x),g_t(x):s<i<t<k\}},\\
g'_i(x)&=\max\left\{0,\tilde{g}_i(x)-\frac{1}{2}\right\}.
\end{align*}

Fix $x\in X$.
By $\Sigma^0_1$-induction, we show that the denominator in the definition of $\tilde{g}_i(x)$ is nonzero.
Note that $g_i(x)>0$ for some $i<k$ since $(U_i)_{i<k}$ covers $X$.
This verifies the base case.
We inductively assume that the denominator of $\tilde{g}_i(x)$ is nonzero, that is, $g'_s(x)>0$ for some $s<i$ or $g_t(x)>0$ for some $t\geq i$.
Suppose that the denominator of $\tilde{g}_{i+1}(x)$ is zero, that is, $g'_s(x)=0$ for any $s\leq i$ or $g_t(x)=0$ for any $t>i$.
Note that $g_i'(x)=0$ implies $\tilde{g}_i(x)\leq 1/2$, and therefore, by definition of $\tilde{g}_i$, we have
\[
g_i(x)\leq \max\{g_s'(x),g_t(x):s<i<t<k\}=0. 
\]

However, this contradicts the induction hypothesis.
Hence, $\langle g'_i,\tilde{g}_i\rangle_{i<k}$ defines a sequence of total continuous functions, and for any $x\in X$, we have $g_i'(x)>0$ for some $i<k$ as seen above.
This means that $W_i=\{x\in X:g_i'(x)>0\}=\{x\in X:\tilde{g}_i(x)>1/2\}$ covers $X$.
Therefore, $F_i=\{x\in X:\tilde{g}_i(x)\geq 1/2\}$ also covers $X$.
Now, if $g_i(x)=0$ then clearly $\tilde{g}_i(x)=0\leq 1/2$; hence we have $W_i\subseteq F_i\subseteq U_i$.
This concludes that $(F_i)_{i<k}$ is a closed shrinking of $(U_i)_{i<k}$.
\end{proof}

\subsubsection{Star refinement}

Let $S\subseteq X$ and $\mathcal{U}$ be a cover of a space $X$.
A {\em star of $S$ w.r.t.~$\mathcal{U}$} is defined as follows:
\[{\rm st}(S,\mathcal{U})=\bigcup\{U\in\mathcal{U}:S\cap U\not=\emptyset\}.\]

We define $\mathcal{U}^\star$ by $\{{\rm st}(U,\mathcal{U}):U\in\mathcal{U}\}$.
A {\em star refinement} of a cover $\mathcal{U}$ of $X$ is a cover $\mathcal{V}$ of $X$ such that $\mathcal{V}^\star$ is a refinement of $\mathcal{U}$.
It is known that a space is normal iff every finite open cover has a finite open star refinement.

\begin{Lemma}[${\sf RCA}_0$]\label{lem:star-refinement}
Let $X$ be a normal space.
Then, for every finite open cover $\mathcal{U}$ of $X$, one can effectively find a finite open star refinement of $\mathcal{U}$. 
\end{Lemma}

\begin{proof}
Given a finite open cover $\mathcal{U}=\{U_i\}_{i<k}$ of $X$, as in the proof of Lemma \ref{lem:coceshrinking}, one can effectively find a closed shrinking $\{F_i\}_{i<k}$ and an open shrinking $\mathcal{W}=\{W_i\}_{i<k}$ such that $W_i\subseteq F_i\subseteq U_i$ for each $i<k$.
Then, $\mathcal{V}_i=\{X\setminus F_i,U_i\}$ is an open cover of $X$.
We define $\mathcal{V}$ as the following open cover of $X$:
\[\mathcal{V}=\mathcal{W}\wedge\bigwedge_{i<k}\mathcal{V}_i:=\left\{W\cap\bigcap_{i<k}V_i:W\in\mathcal{W},\;V_i\in\mathcal{V}_i\right\}.\]

We claim that if $V\in\mathcal{V}$ is of the form $W_{\ell}\cap\bigcap_{i<k}V_i$, then ${\rm st}(V,\mathcal{V})\subseteq U_{\ell}$.
For any $V^*\in\mathcal{V}$ of the form $W_m\cap\bigcap_{i<k}V^*_i$, if $V\cap V^*\not=\emptyset$, then $V^*_\ell\not=X\setminus F_\ell$ since $V\subseteq W_\ell\subseteq F_\ell$.
Therefore, $V^*\subseteq V^*_\ell=U_\ell$.
Consequently, $\mathcal{V}$ is an open star refinement of $\mathcal{U}$ as desired.
\end{proof}

We also define $\mathcal{U}^\triangle$ by $\{{\rm st}(\{x\},\mathcal{U}):x\in X\}$.
A {\em point-star refinement} (or a {\em barycentric refinement}) of  a cover $\mathcal{U}$ of $X$ is a cover $\mathcal{V}$ of $X$ such that $\mathcal{V}^\triangle$ is a refinement of $\mathcal{U}$.
Clearly, every star refinement is a point-star refinement.

\subsubsection{Absolute extensor}

A space $K$ is called an {\em absolute extensor} for a space $X$ if for any continuous map $f\colon P\to K$ on a closed set $P\subseteq X$, one can find a continuous map $g\colon X\to K$ extending $f$, that is, $g\upto P=f\upto P$.
It is known that the topological dimension (and the cohomological dimension) of a normal space can be restated in the context of the absolute extensor.
Classically, it is known that the covering dimension of $X$ is at most $n$ if and only if the $n$-sphere $\mathbb{S}^n$ is an absolute extensor for $X$ (cf.\ \cite[Theorem 1.9.3]{Engelking} or \cite[Theorem III.2]{Nagata}).
This equivalence is due to Eilenberg-Otto.
To prove the equivalence, Eilenberg-Otto introduced the notion of an essential family.


We will need effectivity for inessentiality to reduce the complexity of induction.
Therefore, instead of considering essentiality of a family, we consider the following notion: 
A space $X$ is {\em $(n+1)$-inessential} if for any sequence $(A_i,B_i)_{i<n+1}$ of disjoint pairs of closed sets in $X$, one can effectively find a sequence $(U_i,V_i)_{i<n+1}$ of disjoint open sets in $X$ such that $A_i\subseteq U_i$ and $B_i\subseteq V_i$ for each $i\leq n$ and $(U_i\cup V_i)_{i<n+1}$ covers $X$.

\begin{Lemma}[${\sf RCA}_0$]\label{lem:lemma1}
Let $X$ be a Polish space.
If the $n$-sphere $\mathbb{S}^n$ is an absolute extensor for $X$, then $X$ is $(n+1)$-inessential.
\end{Lemma}

\begin{proof}
As the boundary $\partial {\interval}^{n+1}$ of the $(n+1)$-hypercube $\mathbb{I}^{n+1}$ is homeomorphic to $\mathbb{S}^n$, we can assume that $\partial {\interval}^{n+1}$ is an absolute extensor for $X$.
Given a sequence $(A_i,B_i)_{i<n+1}$ of disjoint pairs of closed sets, one can define $f\colon\bigcup_{i<n+1}(A_i\cup B_i)\to\partial {\interval}^{n+1}$ such that $(\pi_i\circ f)^{-1}\{0\}=A_i$ and $(\pi_i\circ f)^{-1}\{1\}=B_i$ by perfect normality (Fact \ref{fact:perfectly-normal}), where $\pi_i$ is the projection into the $i$th coordinate.
Then, by our assumption, we have $g\colon X\to\partial {\interval}^{n+1}$ which agrees with $f$ on $\bigcup_{i<n+1}(A_i\cup B_i)$.
Define $U_i:=(\pi_i\circ g)^{-1}[0,1/2)$ and $V_i:=(\pi_i\circ g)^{-1}(1/2,1]$.
Then, $(U_i,V_i)_{i<n+1}$ covers $X$ since the range of $g$ is contained in $\partial\mathbb{I}^{n+1}$.
Hence, the sequence $(U_i,V_i)$ witnesses the condition of $(n+1)$-inessentiality.
\end{proof}

\subsubsection{Covering dimension}

Let $\mathcal{U}$ be a cover of a space $X$.
We say that the {\em order of $\mathcal{U}$ is at most $n$} if for any $U_0,U_1,\dots,U_{n+1}\in\mathcal{U}$ we have $\bigcap_{i<n+2}U_i=\emptyset$.
A space $X$ has the {\em covering dimension at most $n$} if for any finite open cover of $X$, one can effectively find a finite open refinement of order at most $n$.

\begin{Lemma}[${\sf RCA}_0$]\label{lem:lemma2}
Let $X$ be a Polish space.
If $X$ is $(n+1)$-inessential, then the covering dimension of $X$ is at most $n$.
\end{Lemma}

\begin{proof}
We first show the following claim.

\begin{Claim}[${\sf RCA}_0$]\label{claim:inessential}
If $X$ is $(n+1)$-inessential, then for any open cover $\mathcal{U}=(U_i)_{i<n+2}$ of $X$, one can effectively find an open shrinking $\mathcal{W}=(W_i)_{i<n+2}$ of $\mathcal{U}$ such that $\bigcap\mathcal{W}=\emptyset$.
\end{Claim}

\begin{proof}
We follow the argument in Engelking \cite[Theorem 1.7.9]{Engelking}.
Given an open cover $\mathcal{U}=(U_i)_{i<n+2}$ of $X$, pick a closed shrinking $(F_i)_{i<n+2}$ by Lemma \ref{lem:coceshrinking}.
Then, consider the sequence $(U_i,X\setminus F_i)_{i<n+1}$ of open covers.
By $(n+1)$-inessentiality,  one can find a sequence of disjoint open sets $(W_i,V_i)_{i<n+1}$ in $X$ such that $W_i\subseteq U_i$, $V_i\subseteq (X\setminus F_i)$ and $\bigcup_{i<n+1}W_i\cup V_i$ covers $X$.
Define $W_{n+1}:=U_{n+1}\cap\bigcup_{i<n+1}V_i$.
As $F_{n+1}\subseteq U_{n+1}$, we have the following:
\[\bigcup\mathcal{W}=\left[\bigcup_{i<n+1}W_i\cup U_{n+1}\right]\cap\left[\bigcup_{i<n+1}W_i\cup\bigcup_{i<n+1}V_i\right]\supseteq\bigcup_{i<n+2}F_i=X.\]

Thus, $\mathcal{W}=(W_i)_{i<n+2}$ is an open cover of $X$.
Moreover, as $V_i$ and $W_i$ are disjoint, we have
\[\bigcap_{i<n+2}W_i=\bigcap_{i<n+1}W_i\cap\left[U_{n+1}\cap\bigcup_{i<n+1}V_i\right]\subseteq\bigcap_{i<n+1}W_i\cap\bigcup_{i<n+1}V_i=\emptyset.\]

This concludes that $\mathcal{W}$ is an open refinement of $\mathcal{U}$ of order at most $n$ as desired.
\end{proof}

We then follow the argument in Engelking \cite[Theorem 1.6.10]{Engelking}.
Suppose that $\mathcal{U}=\{U_i\}_{i<s}$ is a finite open cover of $X$.
Let $[s]^{n+2}$ be the collection of all set $D\subseteq s$ such that $|D|=n+2$, and $D_e$ be the $e$-th element in $[s]^{n+2}$.
Put $b:=|[s]^{n+2}|=\binom{s}{n+2}$.
Set $U^{-1}_i=U_i$.
We will construct a sequence $(F_i^e,U_i^e)_{e<b}$ of pairs of a closed set $F_i^e$ and an open set $U_i^e$ such that $(U^e_i)_{i<s}$ is an open shrinking of $\mathcal{U}$, and moreover,
\begin{align*}
(\forall i<s)\;U^{e}_i\subseteq F^{e}_i\subseteq U^{e-1}_i,\mbox{ and }\bigcap_{i\in D_e}U_i^e=\emptyset.
\end{align*}

Given a sequence $\mathcal{U}=(U_i)_{i<s}$ of open set which is given as cozero sets of $(u_i)_{i<s}$, by Claim \ref{claim:inessential}, one can effectively find a code of a sequence $(w_i)_{i\in D_e}$ of partial continuous functions such that, whenever $\mathcal{U}$ is a cover of $X$, $w_i$ is total, the cozero sets $\mathcal{W}=(W_i)_{i\in D_e}$ of $(w_i)_{i\in D_e}$ are an open shrinking of $(U_i)_{i\in D_e}$, and $\mathcal{U}':=(U_i,W_j:i\in D_e,\,j\not\in D_e)$ covers $X$.
Put $u_i'=u_i$ for $i\not\in D_e$ and $u_i'=w_i$ for $i\in D_e$.
Then, $\mathcal{U}'$ is given as a collection of cozero sets of $u_i'$'s.

Then, by Lemma \ref{lem:coceshrinking}, one can effectively find a code of a sequence $(\tilde{v}_i)_{i<s}$ of partial continuous functions such that, whenever $\mathcal{U}$ is a cover of $X$, $v_i$ is total, $u'_i(x)=0$ implies $\tilde{v}_i(x)=0$, and $V_i=\{x:\tilde{v}_i(x)>1/2\}$ covers $X$.
Put $F_i=\{x:\tilde{v}_i(x)\geq 1/2\}$, and $v_i(x)=\max\{0,\tilde{v}_i(x)-1/2\}$.

It is clear that if $\mathcal{U}$ is an open cover of $X$, then $(V_i)_{i<s}$ is an open shrinking of $\mathcal{U}$, and moreover,
\[V_i\subseteq F_i\subseteq U_i,\mbox{ and }\bigcap_{i\in D_e}V_i=\emptyset.\]

To reduce the complexity of induction, we now note that the construction $(u_i)_{i<s}\mapsto(v_i,\tilde{v}_i)_{i<s}$ is effective, i.e., has an explicit $\Sigma^0_1$-description $\Phi$.
Hence, one can effectively obtain (a code of) a sequence $(\tilde{g}_i^e,g_i^e)_{e,i}$ such that $(u_i)_{i<s}=(g_i^{e-1})$ and $(\tilde{v}_i,v_i)_{i<s}=(\tilde{g}_i^e,g_i^{e})_{i<s}$ satisfies the $\Sigma^0_1$-condition $\Phi$ describing the above construction.
Then, define $U^e_i=\{x:\tilde{g}^e_i(x)>1/2\}$ and $F^e_i=\{x:\tilde{g}^e_i(x)\geq 1/2\}$.

We first check that $(U^e_i)_{i<s}$ forms an open cover for any $e<b$.
Fix $x\in X$.
By $\Sigma^0_1$-induction, one can easily show that for any $e$, $x\in U^e_i$ for some $i<s$.
Next, we see that $U^d_i\subseteq U^e_i$ for any $e\leq d<b$.
Fix $x\in X$.
Note that $g^{e-1}_i(x)=0$ implies $\tilde{g}^e_i(x)<1/2$, and this condition is $\Sigma^0_1$.
For $d>e$, inductively assume that $g^{e-1}_i(x)=0$ implies $\tilde{g}^d_i(x)<1/2$.
Then $\tilde{g}^d_i(x)<1/2$ clearly implies $g^d_i(x)=0$, and therefore, $\tilde{g}^{d+1}_i(x)<1/2$.
By $\Sigma^0_1$-induction, we obtain that $g^{e-1}_i(x)=0$ implies $\tilde{g}^d_i(x)<1/2$ for any $d>e$.
Hence, $g^{e-1}_i(x)=0$ implies $g^d_i(x)=0$ for $d>e$, which implies that $U^d_i\subseteq U^e_i$ for any $e\leq d<b$.

Finally we put $V_i=U^{b-1}_i$.
We have shown that $(V_i)_{i<s}$ is an open shrinking of $\mathcal{U}$.
It remains to show that the order of $(V_i)_{i<s}$ is at most $n$.
To see this, it suffices to show that for any $e$, $\bigcap_{i\in D_e}V_i=\emptyset$.
As shown above, $\mathcal{U}^{e-1}=(U^{e-1}_i)_{i<s}$ forms an open cover.
Therefore, $(U^{e}_i)_{i<s}$ is an open shrinking of $\mathcal{U}^{e-1}$ such that $\bigcap_{i\in D_e}U^e_i=\emptyset$.
Then, as seen before, we have $V_i=U^{b-1}_i\subseteq U^e_i$ for any $i<s$.
Therefore, $\bigcap_{i\in D_e}V_i=\emptyset$ as desired.
\end{proof}

\subsection{N\"obeling's imbedding theorem}

The {\em $n$-dimensional N\"obeling space} $N^n$ is a subspace of ${\interval}^{2n+1}$ consisting of points with at most $n$ rational coordinates.
The N\"obeling imbedding theorem says that an $n$-dimensional separable metrizable space is topologically embedded into the $n$-dimensional N\"obeling space.
We will see that the N\"obeling imbedding theorem is provable in ${\sf RCA}_0$ in the following sense:

\begin{Theorem}[${\sf RCA}_0$]\label{thm:Nobeling}
If the covering dimension of a Polish space $X$ is at most $n$, then $X$ can be topologically embedded into the $n$-dimensional N\"obeling space.

More precisely, there is a topological embedding $f$ of $X$ into $\interval^{2n+1}$ such that for any $x\in X$, at most $n$ coordinates of $f(x)$ are rational.
\end{Theorem}

\subsubsection{The modified Kuratowski mapping}

We say that points $\{p_i\}_{i<\ell}$ in $\mathbb{I}^{d+1}$ are in a {\em general position}, i.e., if $0\leq m\leq d$, then any $m+2$ points from $\{p_i\}_{i<\ell}$ do not lie in an $m$-dimensional hyperplane of $\mathbb{I}^{d+1}$.
The following is an easy effectivization of a very basic observation (cf.\ Engelking \cite[Theorem 1.10.2]{Engelking}).

\begin{Observation}[${\sf RCA}_0$]\label{obs:general-position}
Given $\varepsilon>0$ and points $q_1,\dots,q_k\in \mathbb{R}^m$, one can effectively find $p_1,\dots p_k\in \mathbb{R}^m$ in
general position such that $d(p_i,q_i)<\varepsilon$ for any $i\leq k$.
\qed
\end{Observation}

A {\em polyhedron} is a geometric realization $|\mathcal{K}|$ of a simplicial complex $\mathcal{K}$ in a Euclidean space.
We approximate a given space by a polyhedron as follows:
Let $\mathcal{U}=(U_i)_{i<k}$ be a finite open cover of $X$.
The {\em nerve of $\mathcal{U}$} is an abstract simplicial complex $\mathcal{N}(\mathcal{U})$ with $k$ many vertices $\{p_i\}_{i<k}$ such that an $m$-simplex $\{p_{j_0},\dots,p_{j_{m+1}}\}$ belongs to $\mathcal{N}(\mathcal{U})$ iff $U_{j_0}\cap\dots\cap U_{j_{m+1}}$ is nonempty.
We define the function $\kappa:X\to |\mathcal{N}(\mathcal{U})|$ as follows:
\[\kappa(x)=\frac{\sum_{i=0}^{k-1}d(x,X\setminus U_i)p_i}{\sum_{j=0}^{k-1}d(x,X\setminus U_j)}.\]
The function $\kappa$ is called {\em the $\kappa$-mapping (or Kuratowski mapping) determined by $\mathcal{U}$ and $(p_i)_{i<k}$}.
For basics on the $\kappa$-mapping, see also Engelking \cite[Definition 1.10.15]{Engelking}, and Nagata \cite[Section IV.5]{Nagata}.

However, we cannot ensure the existence of the $(x,i)\mapsto d(x,X\setminus U_i)$ within ${\sf RCA}_0$.
Therefore, we introduce a replacement for the $\kappa$-mapping.
Recall that, within ${\sf RCA}_0$, given an open set $U_i$, one can effectively find a continuous function $u_i\colon X\to[0,1]$ whose cozero set is exactly $U_i$.
The {\em modified $\kappa$-mapping} $\kappa\colon X\to \mathbb{I}^{2n+1}$ determined by $(u_i)_{i<s}$ and $(z_i)_{i<s}$ is defined as follows:
\[\kappa(x)=
\frac{\sum_{i<s}u_i(x)z_i}{\sum_{j<s}u_j(x)}.
\]

The denominator of the above formula is nonzero whenever $\mathcal{U}$ is a cover of $X$.
Given $x\in X$ let $\Lambda(x)$ be the list of all indices $e<s$ such that $x\in U_e$.
Such sets exist by bounded $\Sigma^0_1$ comprehension within ${\sf RCA}_0$.
Let $Z(x)$ be the hyperplane spanned by $(z_e:e\in\Lambda(x))$.

\begin{Claim}[${\sf RCA}_0$]\label{claim:kappa-in-L}
For any $x\in X$, $\kappa(x)$ is contained in the convex hull of $(z_e:e\in\Lambda(x))$, and in particular, $\kappa(x)\in Z(x)$.
\end{Claim}

\begin{proof}
Fix $x\in X$.
By definition of $u_i$, $x\not\in U_i$ (i.e., $i\in\Lambda(x)$) implies $u_i(x)=0$.
Put $\lambda_i=u_i(x)/(\sum_{j\in\Lambda(x)}u_j(x))$.
Clearly, $\sum_{i\in\Lambda(x)}\lambda_i=1$, and $\kappa(x)=\sum_{i\in\Lambda(x)}\lambda_iz_i$.
Hence, $\kappa(x)$ is contained in the convex hull of $(z_e:e\in\Lambda(x))$.
\end{proof}

\subsubsection{Proof of Theorem \ref{thm:Nobeling}}

First note that, to work within ${\sf RCA}_0$, we need to avoid any use of compactness.
Therefore, we cannot use the standard proof of N\"obeling's imbedding theorem.
However, we will see that one can remove compactness arguments from some proof of N\"obeling's imbedding theorem, e.g., given in \cite[Theorem IV.8]{Nagata}, by performing a very careful work.

\begin{proof}[Proof of Theorem \ref{thm:Nobeling}]
For $n+1$ coordinates $(c_i)_{i<n+1}\in (2n+1)^{n+1}$ and $n+1$ rationals $(r_i)_{i<n+1}$, consider the following hyperplane:
\[L=\{(x_j)_{j<2n+1}\in\mathbb{I}^{2n+1}:(\forall i<n+1)\;x_{c_i}=r_i\}.\]

Let $(L_t)_{t\in\om}$ be the list of all such hyperplanes.
For a list $(V_e)_{e\in\om}$ of all basic open balls in $X$, let $\langle i,j\rangle$ be the $t$-th pair such that $\overline{V_i}\subseteq V_j$.
Then, consider the open cover $\mathcal{V}_{t}=\{V_j,X\setminus\overline{V_i}\}$, where $\overline{V_i}$ is the formal closure of $V_i$; that is, the closed ball whose center and radius are the same as $V_i$.

We first gives an explicit construction of (a code of) a sequence $(f_t)_{t\in\mathbb{N}}$ of (possibly partial) continuous functions.
We describe our construction at stage $t$.
Suppose that a continuous function $f_t\colon X\to{\interval}^{2n+1}$ and a positive rational $\delta_t>0$ have already been constructed.
Consider $L_t$ and $\mathcal{V}_t$.
We construct a $\mathcal{V}_t$-mapping $f_{t+1}$ which avoids $L_t$.

By total boundedness of $\mathbb{I}^{2n+1}$, one can easily find a collection $(x_j)_{j\leq m}$ of points in $\mathbb{I}^{2n+1}$ such that $(B(x_j;\delta_t))_{j\leq m}$ covers $\mathbb{I}^{2n+1}$, where $B(x;\delta)$ is the open ball centered at $x$ of radius $\delta$.
Consider $\mathcal{W}_t=\{f_t^{-1}[B(x_j;\delta_t)]:j\leq m\}$.
Since the covering dimension of $X$ is at most $n$, one can effectively find an open refinement of $\mathcal{V}_t\land\mathcal{W}_t$ of order at most $n$.
Apply Lemma \ref{lem:star-refinement} to this new open cover of $X$ to get an open star refinement $\mathcal{U}_t=(U^t_i)_{i<s}$ of $\mathcal{V}_t\land\mathcal{W}_t$ of order at most $n$.
Then, one can effectively find a sequence of continuous functions $(u^t_i)_{i<s}$ such that $U^t_i$ is the cozero set of $u^t_i$.

For each $i<s$, one can effectively choose $x_i\in U^t_i$, and then get the value $f_t(x_i)$.
Then, by Observation \ref{obs:general-position}, we can effectively choose $z^t_i\in X$ and $p^t_j\in L_t$ such that
\begin{align*}
\mbox{$d(f_t(x_i),z^t_i)<\delta$, and $(z^t_i,p^t_j)_{i<s,j<n+1}$ are in a general position,}
\end{align*}
i.e., if $0\leq m\leq 2n$, then any $m+2$ vertices do not lie in an $m$-dimensional hyperplane of $\mathbb{I}^{2n+1}$.
Let $\kappa\colon X\to\mathbb{I}^{2n+1}$ be the modified $\kappa$-mapping determined by $(u_i)_{i<s}$ and $(z_i)_{i<s}$.

\begin{Claim}[${\sf RCA}_0$]\label{claim:uniform-limit}
$d(f_{t}(x),\kappa(x))<3\delta_t$ for any $x\in X$.
\end{Claim}

\begin{proof}
Let $x\in X$ be given.
If $x\not\in U^t_i$, then $u^t_i(x)=0$.
If $x\in U^t_i$, since $\mathcal{U}_t$ is a refinement of $\mathcal{W}_t$, we have $d(f_t(x),f_t(y))<2\delta_t$ for any $y\in U^t_i$.
Therefore, $d(f_t(x),z^t_i)<3\delta_t$ since $d(f_t(x_i),z^t_i)<\delta_t$, where $x_i\in U^t_i$.
Hence, by the definition of the modified $\kappa$-mapping, we get $d(f_t(x),\kappa(x))<3\delta_t$ for any $x\in X$, since
\[d(f_t(x),\kappa(x))=d\left(\sum_{i<s}\lambda_i(x)f_t(x),\sum_{i<s}\lambda_i(x)z^t_i\right)\leq\sum_{i<s}\lambda_i(x)d(f_t(x),z^t_i)<3\delta_t\]
where $\lambda_i(x)$ is defined as in Claim \ref{claim:kappa-in-L}.
The first equality follows from $\sum_{i<s}\lambda_i=1$, and the middle inequality follows from the triangle inequality.
\end{proof}

Let $[s]^{\leq n}$ denote the set of all finite subsets $D\subseteq s$ with $|D|\leq n$, and $Z^t_D$ be the hyperplane spanned by $(z^t_e:e\in D)$.
Now, one can calculate the following value:
\begin{align*}
\eta_t:=\min\{d(Z^t_D,Z^t_E):D,E\in[s]^{\leq n},\;Z^t_D\cap Z^t_E=\emptyset)\}>0.
\end{align*}


Recall that $(z^t_i)_{i\in\Lambda(x)}$ and $(p^t_j)_{j<n+1}$ are in a general position, and $L_t$ is spanned by $(p^t_j)_{j<n+1}$, which implies that $d(Z^t_D,L_t)>0$ for any $D\in[s]^{\leq n}$.
One can also calculate the following value:
\begin{align*}
\eta'_t:=\min\{d(Z^t_D,L_t):D\in[s]^{\leq n}\}>0.
\end{align*}

Now, define $f_{t+1}=\kappa$ (where $\kappa$ is the modified $\kappa$-mapping defined before Claim \ref{claim:uniform-limit}) and $\delta_{t+1}=\min\{\delta_t,\eta_t/8,\eta'_t/4\}/3$.
To reduce the complexity of induction, we now note that the construction $(f_t,\delta_t)\mapsto(f_{t+1},\delta_{t+1},\eta_t,\eta'_t)$ is effective, i.e., has an explicit $\Sigma^0_1$-description.
We then have a sequence $(f_t,\delta_t,\eta_t,\eta'_t)_{t\in\mathbb{N}}$ with auxiliary parameters $(z^t_i)_{t\in\mathbb{N},i<s}$ and $(p^t_j)_{t\in\mathbb{N},j<n+1}$.
A simple induction shows $\delta_t<2^{-t}$.
By $\Sigma^0_1$-induction with Claim \ref{claim:uniform-limit}, for any $t\leq s$, one can also show that $d(f_t(x),f_s(x))<\sum_{s\geq t}\delta_{s+1}<2^{-t}$; hence this is classically a uniform convergent sequence.
Note that the uniform limit theorem is provable within ${\sf RCA}_0$ since a modulus of pointwise continuity of the uniform limit $f=\lim_{t\to\infty}f_t$ is effectively calculated from a sequence of moduli of pointwise continuity of $(f_t)_{t\in\mathbb{N}}$ and the modulus of uniform convergence $2^{-t}$.
Hence, the uniform limit $f=\lim_{t\to\infty}f_t$ exists.
By definition of $\delta_t$, we also get $d(f,f_{t+1})<\eta_t/4,\eta'_t/2$.

\begin{Claim}[${\sf RCA}_0$]\label{claim-v-mapping}
For any $t\in\mathbb{N}$ and $y\in\interval^{2n+1}$ there is $V\in\mathcal{V}_t$ such that $f^{-1}[B(y;\eta_t/4)]\subseteq V$.
\end{Claim}


\begin{proof}
Let $y\in\interval^{2n+1}$ be given.
For $x,x'\in f^{-1}[B(y;\eta_t/4)]$, we have $d(f(x),f(x'))<\eta_t/2$.
As $d(f,f_{t+1})<\eta_t/4$, we have $d(f_{t+1}(x),f_{t+1}(x'))<\eta_t$.
By Claim \ref{claim:kappa-in-L}, we have $f_{t+1}(x)=\kappa(x)\in Z^t(x)$ and $f_{t+1}(x')=\kappa(x')\in Z^t(x')$, where $Z^t(x)$ is defined in a similar manner as before.
By our choice of $\eta_t$, we have $Z^t(x)\cap Z^t(x')\not=\emptyset$.

Assume that $Z^t(x)$ is spanned by $(z^t_{i_\ell})_{\ell<t}$ and $Z^t(x')$ is spanned by $(z^t_{j_\ell})_{\ell<u}$.
Since $Z^t(x)\cap Z^t(x')\not=\emptyset$, $(z^t_{i_{\ell}},z^t_{j_{m}})_{\ell<t,m<u}$ lie on a $((t-1)+(u-1))$-dimensional hyperplane.
By our choice, the open cover $\mathcal{U}_t$ has the order at most $n$, and therefore $t,u\leq n+1$; hence $t+u\leq 2n+2$.
Since $\{z_{i_{\ell}},z_{j_{m}}\}_{\ell<t,m<u}$ are in a general position, $t+u$ vertices do not lie in an $(t+u-2)$-dimensional hyperplane.
Hence, we must have $\ell$ and $m$ such that $z_{i_\ell}=z_{j_m}$.
This implies that $x,x'\in U_{i_\ell}$.

Consequently, if $x,x'\in f_t^{-1}[B(y;\eta_t/4)]$ then $x'$ belongs to the star of $\{x\}$ w.r.t.\ $\mathcal{U}_t$, that is, $x'\in{\rm st}(\{x\},\mathcal{U}_t)$.
As $\mathcal{U}_t$ is a star refinement of $\mathcal{V}_t$, we obtain $V\in\mathcal{V}_t$ such that $f^{-1}[B(y;\eta_t/4)]\subseteq{\rm st}(\{x\},\mathcal{U}_t)\subseteq V$.
\end{proof}



\begin{Claim}\label{claim:line-avoiding}
$d(f(x),p)>\eta_t'/2$ for any $x\in X$ and $p\in L_t$.
\end{Claim}

\begin{proof}
By definition of $\eta_t'$, we have $d(Z_t(x),L_t)\geq \eta_t'$ for any $x\in X$. 
By Claim \ref{claim:kappa-in-L}, we also have $f_{t+1}(x)\in Z_t(x)$, and therefore $d(f_{t+1}(x),L_t)\geq \eta_t'$.
Hence, $d(f(x),L_t)\geq\eta_t'/2$.
\end{proof}

Claim \ref{claim:line-avoiding} ensures that the range of $f$ avoids $L_t$; hence $f$ is a continuous map from $X$ into the $n$-dimensional N\"obeling space $N^n\subseteq\mathbb{R}^{2n+1}$.

\begin{Claim}
$f$ is injective.
\end{Claim}

\begin{proof}
Let $W$ be any open neighborhood of $x\in X$.
Then, by perfect normality of $X$ (Fact \ref{fact:perfectly-normal}), there are basic open balls $V_i$ and $V_j$ such that $x\in V_i\subseteq\overline{V_i}\subseteq V_j\subseteq W$.
By applying Claim \ref{claim-v-mapping} to the code $t$ of pairs $\langle i,j\rangle$ (i.e., $\mathcal{V}_t=\{V_j,X\setminus\overline{V_i}\}$), we get an open neighborhood $B$ of $f(x)$ such that either $f^{-1}[B]\subseteq V_j$ or $f^{-1}[B]\subseteq X\setminus\overline{V_i}$.
However, as $x\in V_i$, we have $x\in f^{-1}[B]\cap V_i\not=\emptyset$; hence $f^{-1}[B]\subseteq V_j$.
Therefore, if $x'\not\in W$ then, as $W\supseteq V_j$, we get $f(x')\not\in B$.
This implies that $f$ is injective.
\end{proof}

It remains to show that $f^{-1}$ is continuous in ${\sf RCA}_0$.
In the usual proof, by using the property that $f$ is an $\varepsilon$-mapping for all $\varepsilon>0$, we conclude that $f$ is a closed map.
However, it is unclear that, from the property being an $\varepsilon$-mapping, how one can effectively obtain a code of the closed image $f[A]$ of a closed set $A\subseteq X$ (without using any compactness arguments).
Fortunately, Claim \ref{claim-v-mapping} has more information than just saying that $f$ is an $\varepsilon$-mapping, which can be used to show that $f$ is an effective open map.

\begin{Claim}\label{claim-open-map}
$f$ is an open map.
\end{Claim}

\begin{proof}
We say that an open ball $B_X(x;q)$ in $X$ is formally (strictly, respectively)\ included in $B_X(y;p)$ if $d(x,y)\leq p-q$ ($d(x,y)<p-q$, respectively).
Note that if $B_X(x;q)$ is strictly included in $B_X(y;p)$ then $\overline{B_X(x;q)}\subseteq B_X(y;p)$.
Let $U=\bigcup_{e}V_{u(e)}\subseteq X$ be an open set given as a union of open balls.
Then, we make a new list $(V_{v(e,j)})_{e,j\in\mathbb{N}}$ of all open balls $V_j$ such that $V_j$ is strictly included in $V_{u(e)}$.

Let $t(e,j)$ be the code of the pair $\langle v(e,j),u(e)\rangle$ (i.e., $\mathcal{V}_{t(e,j)}=\{V_{u(e)},X\setminus\overline{V}_{v(e,j)}\}$).
We now consider a list $(B_k^{e,j})_{k\in\mathbb{N}}$ of all open balls of radius $\leq\eta_{t(e,j)}/4$ in $\mathbb{I}^{2n+1}$.
By Claim \ref{claim-v-mapping}, either $f^{-1}[B_k^{e,j}]\subseteq V_{u(e)}$ or $f^{-1}[B_k^{e,j}]\subseteq X\setminus \overline{V}_{v(e,j)}$ holds.
As we have already seen that $f$ is continuous, we get a code of the open set $f^{-1}[B_k^{e,j}]=\bigcup_{m}V_{s(e,j,k,m)}$.
If we see that $V_{s(e,j,k,m)}$ is formally included in $V_{v(e,j)}$ for some $m$, then we must have $f^{-1}[B_k^{e,j}]\subseteq V_{u(e)}$.
Let $(J_i)_{i\in\mathbb{N}}$ be a list of all such open balls $B_k^{e,j}$, that is,
\[\{J_i\}_{i\in\mathbb{N}}=\{B_k^{e,j}:\mbox{$V_{s(e,j,k,m)}$ is formally included in $V_{v(e,j)}$ for some $m$}\}.\]

We claim that $f[U]=\bigcup_{i\in\mathbb{N}}J_i$.
If $x\in U$, then $x\in V_{u(e)}$ for some $e$, and so $x\in V_{v(e,j)}$ for some $j$.
By Claim \ref{claim-v-mapping}, if $B$ is a sufficiently small basic open ball containing $f(x)$, then $f^{-1}[B]\subseteq V_{v(e,j)}$.
Hence, $f^{-1}[B]$ contains an open ball which is formally included in $V_{v(e,j)}$.
Therefore, $f(x)\in B=J_i$ for some $i\in\mathbb{N}$.
For the converse, if $J_i=B$ then $f^{-1}[B]\subseteq V_{u(e)}$ for corresponding $e$ as mentioned above, and therefore, $f^{-1}[B]\subseteq V_{u(e)}\subseteq U$.
Consequently, $B\subseteq f[U]$.
\end{proof}

By Claim \ref{claim-open-map}, one can effectively obtain a code of $f^{-1}$ as a continuous function.
This concludes the proof.
\end{proof}

\subsection{Every Polish space is at most one dimensional}

We say that $K$ is an {\em absolute extensor} if it is an absolute extensor for any Polish space.
In other words, if $X$ is a Polish space, for any continuous map $f\colon P\to K$ on a closed set $P\subseteq X$, one can find a continuous map $g\colon X\to K$ extending $f$.
The Tietze extension theorem states that the unit interval $\mathbb{I}$ is an absolute extensor.
This clearly implies that $\mathbb{I}^n$ is also an absolute extensor by coordinatewisely extending $f=(f_i)_{i<n}\colon P\to\mathbb{I}^n$ to $g=(g_i)_{i< n}\colon X\to\mathbb{I}^n$.
It is known that the effective version of the Tietze extension theorem is provable within ${\sf RCA}_0$ as follows:

\begin{Fact}[see Simpson {\cite[Theorem II.7.5]{Simpson}}]\label{fact:Tietze}
The Tietze extension theorem is provable in ${\sf RCA}_0$, that is, $\mathbb{I}^n$ is an absolute extensor.
\qed
\end{Fact}

It is intuitively obvious that the topological dimension of the $n$-hypercube ${\interval}^{n}$ is $n$ (but the proof is not so easy even in the classical world).
Surprisingly, however, under $\neg{\sf WKL}$, {\em every Polish space is at most one-dimensional} in the following sense.

\begin{Lemma}[${\sf RCA}_0+\neg{\sf WKL}$]\label{lem:extension-dimension}
If $X$ is a Polish space, then the $1$-sphere $\mathbb{S}^1$ is an absolute extensor for $X$.
\end{Lemma}

\begin{proof}
By Orevkov's construction \cite{Orevkov} (cf.\ Shioji-Tanaka \cite{ShTa90}), if weak K\"onig's lemma fails, then there is a continuous retraction $r\colon{\interval}^2\to\partial {\interval}^2$.
Note that the $1$-dimensional sphere $\mathbb{S}^1$ is homeomorphic to $\partial {\interval}^2$.
Let $f\colon P\to\partial {\interval}^2$ be a continuous map on a closed set $P\subseteq X$.
Then, since ${\interval}^2$ is an absolute extensor by Fact \ref{fact:Tietze}, one can effectively find a continuous extension $f^*\colon X\to {\interval}^2$ of $f$ such that $f^*\upto P=f\upto P$.
Then $g=r\circ f^*\colon X\to\partial\mathbb{I}^2$ is continuous and extends $f$ since $r$ is a continuous retraction.
This concludes that $\mathbb{S}^1$ is an absolute extensor for $X$ as $\mathbb{S}^1\simeq\partial\mathbb{I}^2$.
\end{proof}

\begin{proof}[Proof of Theorem \ref{thm:main-theorem} (4)$\Rightarrow$(1)]
Suppose $\neg{\sf WKL}$.
Then, by Lemma \ref{lem:extension-dimension}, $\mathbb{S}^1$ is an absolute extensor for $\mathbb{R}^m$.
By Lemmata \ref{lem:lemma1} and \ref{lem:lemma2}, the covering dimension of $\mathbb{R}^m$ is at most one.
By Theorem \ref{thm:Nobeling}, there is a topological embedding $f$ of $\mathbb{R}^m$ into the one-dimensional N\"obeling space; that is, for any $x\in\mathbb{R}^m$, at most one coordinate of $f(x)\in\mathbb{R}^3$ is rational.
Consequently, there is a topological embedding of $\mathbb{R}^m$ into $\mathbb{R}^3$.
\end{proof}


%

\section{Continuous degrees}

In this section, we mention some relationship between reverse mathematics of topological dimension theory and J.\ Miller's work on continuous degrees \cite{Miller}.

Classically, a space is countable dimensional if it is a countable union of zero dimensional subspaces.
However, within ${\sf RCA}_0$, it is difficult to handle with the notion of a subspace.
Instead, we use the following definition.
A {\em copy of a subspace of $Y$ in $X$} is a pair $S=(f,g)$ of (codes of) partial continuous functions $f\pcolon X\to Y$ and $g\pcolon Y\to X$.
Then, we say that {\em $x\in X$ is a point in $S=(f,g)$} if $f(x)$ is defined, and $g\circ f(x)$ is defined and equal to $x$.
A separable metric space $X$ is {\em countable dimensional} if $X$ is a union of countably many copies of subspaces of $\mathbb{N}^\mathbb{N}$; that is, there is a sequence $(S_e)_{e\in\mathbb{N}}$ of copies of subspaces of $\mathbb{N}^\mathbb{N}$ such that every $x\in X$ is a point in $S_e$ for some $e\in\mathbb{N}$.

\begin{Theorem}\label{thm:countable-dimensional}
The following are equivalent over ${\sf RCA}_0$:
\begin{enumerate}
\item Weak K\"onig's lemma.
\item The Hilbert cube $\mathbb{I}^\mathbb{N}$ is not countable dimensional.
\end{enumerate}
\end{Theorem}

\begin{proof}
(1)$\Rightarrow$(2):
The usual argument (cf.\ \cite[Theorem 1.8.20]{Engelking}) only uses the Brouwer fixed point theorem, which can be carried out in ${\sf WKL}_0$ \cite{ShTa90}.

(2)$\Rightarrow$(1):
As $\mathbb{I}^\mathbb{N}$ is Polish, if we assume $\neg{\sf WKL}$ then, by Lemma \ref{lem:extension-dimension}, $\mathbb{S}^1$ is absolute extensor for $\mathbb{I}^\mathbb{N}$.
Therefore, by Lemmata \ref{lem:lemma1} and \ref{lem:lemma2}, and Theorem \ref{thm:Nobeling}, $\mathbb{I}^\mathbb{N}$ can be embedded into the $1$-dimensional N\"obeling space $N^1$.
Now, it is clear that $N^1$ is a finite union of zero dimensional subspaces.
\end{proof}

Indeed, the instance-wise version of Theorem \ref{thm:countable-dimensional} holds.
We now consider the instance-wise version in an $\om$-model $(\om,\mathcal{S})$ of ${\sf RCA}_0$:
For (1)$\Rightarrow$(2), if $(S_e)_{e\in\om}\in\mathcal{S}$ is a sequence of copies of subspaces of $\om^\om$, then there is an infinite binary tree $T\in\mathcal{S}$ such that every infinite path through $T$ computes a point $x\in\mathbb{I}^\om$ which is not a point of $S_e$ for any $e\in\om$.
For (2)$\Rightarrow$(1), if $T\in\mathcal{S}$ is an infinite binary tree, then there is a sequence $(S_e)_{e\in\om}\in\mathcal{S}$ of copies of subspaces of $\om^\om$ such that if $x\in\mathbb{I}^\om$ is not a point in $S_e$ for any $e\in\om$, then $x$ computes an infinite path through $T$.

We now interpret this instance-wise $\om$-model version of Theorem \ref{thm:countable-dimensional} in the context of continuous degrees.
We say that {\em $\mathbf{b}$ is PA-above $\mathbf{a}$} (written $\mathbf{a}\ll\mathbf{b}$) if for any $\mathbf{a}$-computable infinite binary tree has a $\mathbf{b}$-computable infinite path.
Miller \cite{Miller} reduced the first-order definability of PA-aboveness to that of continuous degrees:
Whenever $\mathbf{a}$ and $\mathbf{b}$ are total degrees, $\mathbf{a}\ll\mathbf{b}$ if and only if there is a non-total continuous degree $\mathbf{v}$ such that $\mathbf{a}<\mathbf{v}<\mathbf{b}$.
For continuous and total degrees, see Miller \cite{Miller}.

(1)$\Rightarrow$(2) implies Theorem 8.2 in \cite{Miller}: If $\mathbf{a}$ and $\mathbf{b}$ are total degrees and $\mathbf{b}\ll\mathbf{a}$, then there is a non-total continuous degree $\mathbf{v}$ with $\mathbf{b}<\mathbf{v}<\mathbf{a}$.
To see this, consider the topped $\om$-model of ${\sf RCA}_0$ consisting of all sets of Turing degree $\leq\mathbf{b}$.
Then, as in Kihara-Pauly \cite{KiPa}, take the list $(f_e,g_e)$ of all pairs of Turing reductions (more precisely, all reductions in the sense of representation reducibility), which is considered as copies in $\mathbb{I}^\om$ of subspaces of $\om^\om$.
By (1)$\Rightarrow$(2), there is an infinite binary tree $T$ of Turing degree $\mathbf{b}$ such that any path computes $x\in\mathbb{I}^\om$ which is not a point in $(f_e,g_e)$.
Such an $x$ is non-total since there is no $\alpha\in\om^\om$ such that $f_e(x)=\alpha$ and $g_e(\alpha)=x$.
As $\mathbf{b}\ll\mathbf{a}$, such an $x$ is computable in $\mathbf{a}$.
If necessary, by adding a new coordinate to $x$ to code $\mathbf{b}$, we can conclude that there is a non-total degree $\mathbf{v}$ with $\mathbf{b}<\mathbf{v}<\mathbf{a}$.

(2)$\Rightarrow$(1) implies Theorem 8.4 in \cite{Miller}: If $\mathbf{v}$ is a non-total continuous degree and $\mathbf{b}<\mathbf{v}$ is total, then there is a total degree $\mathbf{c}$ with $\mathbf{b}\ll\mathbf{c}<\mathbf{v}$.
To see this, consider the same $\om$-model $\mathcal{S}$ as above.
As in Kihara-Pauly \cite{KiPa}, we consider a copy $S\in\mathcal{S}$ of a subspace of $\om^\om$ in $\mathbb{I}^\om$ as a pair of $\mathbf{b}$-relative Turing reductions.
As $\mathbf{v}$ is non-total, and $\mathbf{b}\leq\mathbf{v}$, a point $x\in\mathbb{I}^\om$ of degree $\mathbf{v}$ avoids any sequence of copies $(S_e)_{e\in\om}\in\mathcal{S}$ of subspaces of $\om^\om$ in $\mathbb{I}^\om$.
Hence, by (2)$\Rightarrow$(1), for any infinite binary tree $T\in\mathcal{S}$, $x$ computes an infinite path $c$ through $T$.
Consequently, we have $\mathbf{b}\ll\mathbf{c}<\mathbf{v}$ for some $\mathbf{c}$.

This argument indicates that (some of) J.\ Miller's work \cite{Miller} (on definability of PA-degrees via continuous degrees) can be considered as the computable instance-wise version of Theorem \ref{thm:countable-dimensional}.

\begin{Acknowledgement}
The author's research was partially supported by JSPS KAKENHI Grant 17H06738, 15H03634, the JSPS Core-to-Core Program (A. Advanced Research Networks), and the Young Scholars Overseas Visit Program in Nagoya University.
The author would like to thank Keita Yokoyama for valuable discussions.
\end{Acknowledgement}

\bibliographystyle{plain}
\bibliography{Brouwer}
\end{document}